\documentclass[x11names,11pt,twoside]{amsart}
\usepackage[utf8]{inputenc}
\usepackage[english]{babel}
\usepackage[T1]{fontenc}
\usepackage{amsfonts}
\usepackage{amsmath, amsthm, amssymb}
\usepackage{latexsym}
\usepackage{mathtools}
\usepackage{bm}
\usepackage{tikz}
\usepackage{tikz-cd}
\usepackage{tkz-euclide}
\usepackage[bookmarks=true]{hyperref}
\usepackage{cleveref}
\Crefname{counterexample}{counterexample}{counterexamples}
\Crefname{counterexample}{Counterexample}{Counterexamples}
\Crefname{conjecture}{conjecture}{conjectures}
\Crefname{conjecture}{Conjecture}{Conjectures}

\newtagform{roman}[]()

\theoremstyle{plain}
\newtheorem{theorem}{Theorem}[section]
\newtheorem{proposition}[theorem]{Proposition}
\newtheorem{corollary}[theorem]{Corollary}
\newtheorem{lemma}[theorem]{Lemma}
\theoremstyle{definition}

\newtheorem{conjecture}[theorem]{Conjecture}
\newtheorem{remark}[theorem]{Remark}

\newcommand{\R}{\mathbb{R}}
\newcommand{\Z}{\mathbb{Z}}
\newcommand{\N}{\mathbb{N}}
\newcommand{\conv}{\mathrm{conv}}
\newcommand{\width}{\mathrm{width}}

\newcommand{\vol}{\operatorname{vol}}
\newcommand{\Flt}{\operatorname{Flt}}

\newcommand{\bt}{\mathbf t}
\newcommand{\bs}{\mathbf s}

\newcommand{\bo}{\mathbf 0}
\newcommand{\sh}[1]{s_{#1}^{(h)}}
\newcommand{\sv}[1]{s_{#1}^{(v)}}

\title[A local maximizer for lattice width of $3$-dim hollow bodies]{A local maximizer for lattice width of $3$-dimensional hollow bodies}

\author[G. Averkov, G. Codenotti, A. Macchia, F. Santos]{Gennadiy Averkov, Giulia Codenotti, Antonio Macchia, Francisco Santos}

\address{{\small G.~Averkov, Brandeburg University of Technology, Platz der Deutschen Einheit 1, 03046 Cottbus, Germany}}
\email{averkov@b-tu.de}

\address{{\small G.~Codenotti, Goethe-Universit\"at, FB 12 – Institut f\"ur Mathematik, Postfach 11 19 32, D–60054 Frankfurt am Main, Germany}}
\email{codenotti@math.uni-frankfurt.de}

\address{{\small A.~Macchia, Fachbereich Mathematik und Informatik, Freie Universit\"at Berlin, Arnimallee 2, 14195 Berlin, Germany}}
\email{macchia@zedat.fu-berlin.de}

\address{{\small F.~Santos, Depto.~de Matem\'aticas, Estad\'istica y Computaci\'on, Universidad de Cantabria, Spain
}}
\email{francisco.santos@unican.es}

\thanks{The second, third and fourth authors were supported by the Einstein Foundation Berlin under grant EVF-2015-230.
 Work of F. Santos is also supported by project MTM2017-83750-P of the Spanish Ministry of Science (AEI/FEDER, UE)}

\usepackage[colorinlistoftodos]{todonotes}

\begin{document}

\begin{abstract}
The second and fourth authors have conjectured that a certain hollow tetrahedron $\Delta$ of width $2+\sqrt2$ attains the maximum lattice width among all three-dimensional convex bodies.
We here prove a local version of this conjecture: there is a neighborhood $U$ of $\Delta$ in the Hausdorff distance such that every convex body in $U\setminus\{\Delta\}$ has width strictly smaller than $\Delta$. When the search space is restricted to tetrahedra, we compute an explicit such neighborhood.

We also limit the space of possible counterexamples to the conjecture. We show, for example, that their width must be smaller than $3.972$ and their volume must lie in $[2.653, 19.919]$.
\end{abstract}

\maketitle

\section{Introduction}
In the paper \cite{CodSan}, the second and fourth authors explore lower bounds for the lattice width of hollow convex bodies.
Recall that a convex body $K \subset \R^d$ is \emph{hollow} with respect to an affine $d$-dimensional lattice $\Lambda \subset \R^d$ if $\Lambda$ does not intersect the interior of $K$.
The width of $K$ in the direction of a linear functional $f$, denoted $\width(K,f)$, is the length of the segment $f(K)$.
The \emph{(lattice) width} of $K$ with respect to $\Lambda$, denoted $\width_{\Lambda}(K)$, is the minimum width with respect to all non-zero lattice functionals in  $\vec \Lambda^*$, the linear lattice dual to $\vec \Lambda :=\Lambda-\Lambda$.
The famous Flatness Theorem says that the so-called \emph{flatness constant},
\[
	\Flt(d) :=  \max \{ \width_{\Lambda}(K) \, :\, K \subset \R^d \ \text{hollow with respect to} \ \Lambda \},
\]
is a finite value.
The flatness constant does not depend on $\Lambda$, because one can change coordinates and fix $\Lambda = \Z^d$. Still, in our considerations it will make sense to allow for other lattices to highlight symmetry in the constructions.

Much work has been done in improving asymptotic upper bounds for $\Flt(d)$, from Khinchine's original proof in 1948 \cite{Khinchine} to the current best bound
\[
\Flt(d) \in O(d^{4/3} \log^a d),
\]
for a constant $a$, proved by Rudelson in 2000 \cite{Rudelson}.\footnote{Rudelson's paper does not mention $\Flt(d)$ explicitly, but the asserted bound follows from Rudelson's main result by applying the so-called $M M^\ast$ estimate to the flatness constant, as explained in \cite[p. 729, l.~19]{Banaszczyk_etal}}
See also~\cite{Banaszczyk_etal, KanLov} or \cite[Sect.~VII.8]{Barvinok} for more information, and \cite{Dash_etal} for applications to linear programming in fixed dimension.
Finding the exact value of $\Flt(d)$ is a remarkably hard task, even in very low dimensions. Currently, we only know that $\Flt(1)=1$, which is trivial, and $\Flt(2) = 1  + 2 / \sqrt{3}$, which is a result of Hurkens \cite{Hurkens} (see also \cite{AveWag}). In \cite{CodSan} the inequality  $\Flt(3) \ge 2 + \sqrt{2}$ is proven by introducing
 the tetrahedron
$
\Delta= \conv(a_1,a_2,a_3,a_4),
$
where
\begin{align}
\begin{array}{ll}
a_1=\big( 2+ {\sqrt2},  {\sqrt2},  2+ {\sqrt2} ),&
a_2=\big(- {\sqrt2}, 2+{\sqrt2},  -2- {\sqrt2} ),\\
a_3=\big(-2- {\sqrt2}, -{\sqrt2},  2+ {\sqrt2} ),&
a_4=\big( {\sqrt2}, -2-{\sqrt2},  -2- {\sqrt2} )
\label{eq:Delta}
\end{array}
\end{align}
(see \Cref{fig:width3.4}).
With respect to the affine lattice
\[
\Lambda:= \left\{ (a,b,c) : a,b,c \in 1+2\Z, a+b+c \in 1 + 4\Z \right\},
\]
$\Delta$ is hollow and has width $2+\sqrt2$. More precisely, $\Delta$ attains that width with respect to seven different functionals in $\vec \Lambda^*$, namely:
{\large
\begin{align}
\label{eq:extr_functionals}
\begin{array}{llll}
\frac14(1,1,1),& \qquad
\frac14(-1,1,1), &\qquad
\frac14(1,1,-1), &\qquad
\frac14(1,-1,1), \\
&\qquad
\frac12(1,0,0), &\qquad
\frac12(0,1,0), &\qquad
\frac12(0,0,1).
\end{array}
\end{align}
}
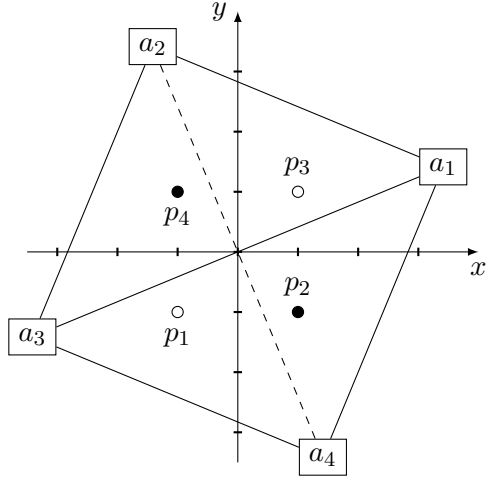
\begin{figure}[htb]
  \centering
\begin{tikzpicture}[bnode/.style={circle, inner sep=1.5pt, fill=black},
wnode/.style={circle, inner sep=1.5pt, fill=white},
scale =  .8]
  \node[draw,bnode, label=below:{$p_4$}] () at (-1, 1) {};
  \node[draw,wnode, label=above:{$p_3$}] () at (1, 1) {};
  \node[draw,wnode, label=below:{$p_1$}] () at (-1, -1) {};
  \node[draw,bnode, label=above:{$p_2$}] () at (1, -1) {};

   \tkzInit[xmax=3.5,ymax=3.5,xmin=-3.5,ymin=-3.5]
        \tkzDrawX
        \tkzDrawY

  \node[draw, label=right:{}] (a_1) at (3.4142, 1.4142) {$a_1$};
  \node[draw, label=above:{}] (a_2) at (-1.4142, 3.4142) {$a_2$};
  \node[draw, label=left:{}] (a_3) at (-3.4142, -1.4142) {$a_3$};
  \node[draw, label=below:{}] (a_4) at (1.4142, -3.4142) {$a_4$};
  \draw[black] (a_1) -- (a_2) -- (a_3) -- (a_4) -- (a_1) -- (a_3);
  \draw[dashed] (a_2) -- (a_4);
\end{tikzpicture}

\caption{A hollow $3$-simplex $\Delta$ of width $2+\sqrt2$}
\label{fig:width3.4}
\end{figure}
As indicated in \Cref{fig:width3.4}, each facet of $\Delta$ contains a single lattice point. More precisely, for each permutation $(i,j,k,l)$ of $\{1,2,3,4\}$ the facet spanned by $\{a_i,a_j,a_k\}$ contains the lattice point $p_l$ with the following coordinates:
\begin{align*}
p_1=(-1,-1,-1)  , \quad&
p_2=(1,-1,1)   , \quad&
p_3=(1,1,-1)  , \quad&
p_4=(-1,1,1).
\end{align*}
These four lattice points form an affine lattice basis for $\Lambda$.

\begin{conjecture}[{\cite[Conjecture 1.2]{CodSan}}]
\label{conj:maximal}
No hollow convex $3$-body has width larger than $\Delta$.
That is, $\Flt(3) = 2+ \sqrt{2}$.
\end{conjecture}

We here give evidence for this conjecture, proving a local version of it:

\begin{theorem}
\label{theorem:main}
The tetrahedron $\Delta$ is a strict local maximizer for width among hollow tetrahedra. That is, every small perturbation of $\Delta$ is either non-hollow or has width strictly smaller than $2+\sqrt2$.
\end{theorem}

\begin{corollary}
\label{coro:main}
The tetrahedron $\Delta$ is a strict local maximizer for width among hollow convex $3$-bodies. More precisely,
there exists an open set $U\subset \R^3$ containing  $\Delta$ and such that every hollow convex body $K \subset U$ different from $\Delta$ has width strictly smaller than $2 + \sqrt{2}$.
\end{corollary}

\begin{proof}
Since $0$ lies in the interior of $\Delta$, we can choose $U$ to be the interior of $(1+ \epsilon) \Delta$, where $\epsilon>0$ is small enough. Let $K$ be a hollow convex body contained in $U$.

Consider the functionals $f_1:=x/2$ and  $f_2:=y/2$, which are in $\Lambda^*$ and both give width $2+\sqrt2$ to $\Delta$. As seen in \Cref{fig:width3.4}, $f_1$ attains its unique minimum and maximum over $\Delta$ at $a_3$ and $a_1$, while $f_2$ attains them at $a_4$ and $a_2$.
Hence, in order for $\width(K,f_i) \ge 2 + \sqrt{2}$ for both functionals, $K$ must contain points at distance $O(\epsilon)$ to each of the four vertices of $\Delta$. Without loss of generality, in the rest of the proof we assume this to happen.

Consider now the lattice points $p_1,\dots,p_4$, lying respectively in the relative interior of the four facets of $\Delta$.
Since $K$ is hollow, there are planes $H_1,\dots,H_4$ weakly separating $K$ from them. The fact that $K$ is contained in a small neighborhood of $\Delta$ and that it contains points close to all vertices of $\Delta$ implies that the inequalities for $H_1,\dots,H_4$ are close to the facet inequalities of $\Delta$. Thus, by \Cref{theorem:main}, the tetrahedron $\Delta'$ defined by $H_1,\dots,H_4$ has width bounded by $2+\sqrt2$, with equality if and only if $\Delta'=\Delta$.

Since $\Delta'$ contains $K$, we conclude that also $K$ has its lattice width bounded by $2+\sqrt2$ and that equality can only occur if $\Delta'=\Delta$, which implies $K\subset\Delta$. If this containment is strict, then $\width(K) < \width(\Delta)$ because of the functionals $f_1$ and $f_2$.
\end{proof}

In \Cref{sec:setting} we transform \Cref{theorem:main} into a more explicit \Cref{theorem:6functionals}, which is then proved in \Cref{sec:neighborhood} with a method based on the Karush-Kuhn-Tucker (KKT) conditions for local optimality.

\Cref{sec:explicit} is devoted to computing an explicit neighborhood of $\Delta$ in the space of all tetrahedra where $\Delta$ is guaranteed to be the unique maximizer for width. In order to do this, it is convenient to think of $\Delta$ as fixed and let the lattice $\Lambda$ vary, instead of the opposite. That is, for any choice of points $p'_1,\dots,p'_4$ lying respectively in the facet of $\Delta$ containing the corresponding $p_1, \dots, p_4$, we consider the affine lattice $\Lambda'$ generated by $p'_1,\dots,p'_4$. The problem can now be reformulated as asking how far each $p'_i$ can be from $p_i$ while maintaining $\width_{\Lambda'}(\Delta) \le\width_{\Lambda}(\Delta)$.
A natural measure for the distance between $p_i$ and $p'_i$ is the $L_\infty$ distance between their vectors of barycentric coordinates in $\Delta$; we denote this $\operatorname{dist}_\Delta(p_i, p'_i)$. Our main result in \Cref{sec:explicit} can then be stated as follows (see also \Cref{thm:explicit} in \Cref{subsec:bs}, which uses a slightly different metric):

\begin{theorem}
\label{thm:barycentric}
Let $p'_1,\dots,p'_4$ be points lying respectively in the same facet of $\Delta$ as $p_1,\dots,p_4$ and let $\Lambda'$ be the affine lattice spanned by $p'_1,\dots,p'_4$.
If $\operatorname{dist}_\Delta(p_i, p'_i) \le 0.01307$ for every $i$, then
\[
\width_{\Lambda'}(\Delta) \le \width_{\Lambda}(\Delta),
\]
with equality if and only if $(p'_1,\dots,p'_4) = (p_1,\dots,p_4)$.
\end{theorem}

In \Cref{sec:global} we explore properties that potential convex $3$-bodies of width larger than $2+\sqrt2$ must satisfy. That is to say, we try to specify a search space for validating (or refuting) \Cref{conj:maximal}. Our main results are that such bodies must have:
\begin{itemize}
\item  width bounded above by $3.972$, that is, we prove $\Flt(3) < 3.972$;
\item  volume between $2.653$ and $19.919$;
\item  an inscribed lattice polytope that is either a unimodular quadrilateral or an empty $3$-polytope of volume bounded above by $22/3$.
\end{itemize}
Recall that a \textit{lattice polytope} is a polytope with vertices in the lattice, and it is called \emph{empty} if its only lattice points are its vertices. Whenever we refer to volume, we normalize it to a fundamental domain of the lattice.

\section{Setting the problem}
\label{sec:setting}

To prove \Cref{theorem:main} we find more convenient to look at perturbations of the lattice, keeping $\Delta$ fixed, rather than the other way around.
That is to say, we fix $\Delta$ to have the vertex coordinates of \Cref{eq:Delta} and we let
$\Lambda(\mathbf t)$ be the affine lattice generated by:
\begin{align*}
p_1(\mathbf t)=(-1,-1,-1)  + (t_{11}, t_{12}, t_{13}), \qquad&
p_2(\mathbf t)=(1,-1,1)  + (t_{21}, t_{22}, t_{23}), \\
p_3(\mathbf t)=(1,1,-1)  + (t_{31}, t_{32}, t_{33}), \qquad&
p_4(\mathbf t)=(-1,1,1) + (t_{41}, t_{42}, t_{43}),
\end{align*}
where the $t_{ij}$'s are variables.
Observe that $\Lambda(\mathbf 0)=\Lambda$.
Our task is to study the width of $\Delta$ with respect to $\Lambda(\mathbf t)$ as a function of $\mathbf t$ and to show that $\mathbf 0$ is a strict local maximizer of it, under the constraint that $\Delta$ is hollow.

Since a tetrahedron of maximal width necessarily has at least one lattice point on (the relative interior of) every facet, and since the facets of $\Delta$ contain each a single point of $\Lambda$, there is no loss of generality in constraining the variables $t_{ij}$ to values where we have the coplanarities  $a_1a_2a_3p_4$, $a_1a_2p_3a_4$, $a_1p_2a_3a_4$ and $p_1a_2a_3a_4$. In practice this means we can express the $t_{*3}$'s in terms of the $t_{*1}$'s and $t_{*2}$'s as follows:
\begin{gather*}
 t_{13}=-\frac{(2+\sqrt2) t_{11}+\sqrt2 t_{12}}{2}, \qquad t_{23}=\frac{-\sqrt2 t_{21}+(2+\sqrt2)t_{22}}{2},\\
 t_{33}=\frac{(2+\sqrt2)t_{31}+\sqrt2 t_{32}}{2}, \qquad t_{43}=\frac{\sqrt2 t_{41}-(2+\sqrt2)t_{42}}{2}.
\end{gather*}
Thus, in what follows we denote
\[
\mathbf t:= (t_{11}, t_{12}, t_{21}, t_{22}, t_{31}, t_{32}, t_{41}, t_{42})
\]
our vector of only eight variables.

The seven functionals of Eq.~\eqref{eq:extr_functionals} need to be perturbed in order to still be lattice functionals for $\Lambda(\mathbf t)$. To derive their exact form we first rewrite them (at $\mathbf t=0$) with respect to the lattice basis $\{p_4-p_1, p_2 -p_1, p_3-p_1\}$, and translate them to vanish at $p_1$. When this is done the functionals become the scalar product with the following vectors:
{\large
\begin{align}
\begin{array}{llll}
u_1=(1,1,1), &
u_2=(1,0,0), &
u_3=(0,0,1), &
u_4=(0,1,0),  \\
&u_5=(0,1,1), &
u_6=(1,0,1), &
u_7=(1,1,0).
\end{array}
\label{eq:u}
\end{align}
}

Now consider the $3\times 3$ matrix
\[
M (\mathbf t)
=
\left[
\begin{matrix}
p_4 (\mathbf t)-p_1 (\mathbf t) \\ p_2 (\mathbf t)-p_1 (\mathbf t) \\ p_3 (\mathbf t)-p_1 (\mathbf t)
\end{matrix}
\right]
\]
as a function of $\mathbf t$. The rows of $M$ form a basis for the perturbed linear lattice $\vec \Lambda(\mathbf t)=\Z^3 \cdot M(\mathbf t)$, so the columns of its inverse $N(\mathbf t):=  M(\mathbf t)^{-1}$ form the corresponding dual basis in $\vec \Lambda(\mathbf t)^* = M(\mathbf t)^{-1} \cdot \Z^3$.
Thus, the perturbed lattice functionals can be written as $M(t)^{-1} u_i$, for the vectors $u_i \in \Z^3$ displayed in \Cref{eq:u}.

The first six functionals attain their maximum and minimum value, for $\mathbf t = \mathbf 0$, at unique vertices of $\Delta$. By continuity, the perturbed functionals will attain their maximum and minimum at the same vertices, for any $\mathbf t$ close to $\mathbf 0$. That is, there is a neighborhood $U_f$ of $\mathbf 0$ such that the width of $\Delta$ with respect to the first six perturbed functionals equals
	\[
		f_i(\bt) := \width(\Delta, M(\bt)^{-1} u_i) = v_i M(\bt)^{-1} u_i,
	\]
where
{\large
\begin{align*}
\begin{array}{ll}
&v_1= a_1-a_4, \qquad
v_2=a_3-a_4, \qquad
v_3=a_2-a_3, \\
&v_4=a_1-a_2, \qquad
v_5=a_1-a_3, \qquad
v_6=a_2-a_4
\end{array}
\end{align*}
}%
are the unique vertices of the difference body $\Delta -\Delta$ where the width is attained for the respective $M(\bt)^{-1} u_i$, $i\in \{1,\dots,6\}$.
The width of $\Delta$ with respect to $M(\bt)^{-1} u_7$ is difficult to express because in $\mathbf0$ this functional attains its maximum at two of the vertices of $\Delta$ ($a_1$ and $a_3$) and its minimum at the other two ($a_2$ and $a_4$). In the rest of the paper we neglect this functional, which is no loss of generality.

Summing up, \Cref{theorem:main} follows from the following statement, which we prove in
\Cref{sec:neighborhood}:

\begin{theorem}
\label{theorem:6functionals}
The system of $6$ inequalities in eight variables
\[
f_i(\bt) \ge 2+\sqrt2, \qquad i\in \{1,\dots,6\}
\]
has an isolated solution at $\mathbf t=\mathbf 0$.
\end{theorem}

\section{A proof of local maximality}
\label{sec:neighborhood}

To prove \Cref{theorem:6functionals}, we want to show that for any $\mathbf t$ close to (but different from) $\mathbf 0$, we have
\[
f_i(\mathbf t) = v_i M(\mathbf t)^{-1} u_i <(2+\sqrt2)
\]
for some $i \in \{1, \dots 6\}$. Our proof involves some computations, that we have performed in exact arithmetics using \textit{SageMath} \cite{Sage}.

The entries of the matrix $M(\mathbf t)^{-1}$ are rational functions in $\mathbf t$.
In order to simplify computations we multiply both sides of our inequalities by $\det M(\mathbf t)$. This does not change their direction (in a neighborhood of $\mathbf 0$) since $\det M(\mathbf 0) = 16 >0$.
Thus, we define the following functionals
	\begin{align*}
		h_i(\mathbf t) := & \det M(\mathbf t) \cdot \left(f_i(\mathbf t)- (2+\sqrt2)\right)
		 \\ = & v_i M(\mathbf t)^{\#} u_i - (2 + \sqrt{2}) \det M(\mathbf t), \qquad i\in \{1,\dots,6\}.
	\end{align*}
Here $M(\mathbf t)^{\#} = \det M(\mathbf t) M(\mathbf t)^{-1}$
is the adjugate matrix, the transpose of the cofactor matrix. The entries of $M(\mathbf t)^{\#}$ are  polynomials of degree $2$ in $\mathbf t$, while the entries of $\det M(\mathbf t)$ are polynomials of degree  $3$ in $\mathbf t$. We thus see that $h_i(\mathbf t)$ is a polynomial of degree $3$ in $\mathbf t$, and our original inequalities $f_i(\mathbf t) \ge 2+ \sqrt2$ are (in a neighborhood of the origin) equivalent to the polynomial inequalitites $h_i(\mathbf t) \ge 0$.

The gradients of the functions $h_1 ,h_2,\dots,h_6$, evaluated at $\mathbf 0$, are:
\begin{align}
\small
\begin{array}{ll}
\nabla h_1(\mathbf 0) = 4(-1,1,-1,-2,0,0,-2,1) &+\ 2\sqrt2(-2,0,1,-1,0,0,-3,1) \\
\nabla h_2(\mathbf 0) = 4(-2,1,0,0,1,2,1,1) &+\ 2\sqrt2(-1,-1,0,0,1,3,0,2) \\
\nabla h_3(\mathbf 0) = 4(0,0,2,-1,1,-1,1,2) &+\ 2\sqrt2(0,0,3,-1,2,0,-1,1)\\
\nabla h_4(\mathbf 0) = 4(-1,-2,-1,-1,2,-1,0,0) \hspace{-5pt}\null&+\ 2\sqrt2(-1,-3,0,-2,1,1,0,0)\\
\nabla h_5(\mathbf 0) = 8(1,0,-1,0,-1,0,1,0) &+\ 8\sqrt2(1,0,0,0,-1,0,0,0)\\
\nabla h_6(\mathbf 0) = 8(0,1,0,1,0,-1,0,-1) &+\ 8\sqrt2(0,0,0,1,0,0,0,-1)
\end{array}
\label{eq:nablas}
\end{align}

These six vectors happen to have rank five, with the following positive linear dependence among them:
\begin{gather*}
    \nabla h_1(\mathbf 0)+\nabla h_2(\mathbf 0)+\nabla h_3(\mathbf 0)+\nabla h_4(\mathbf 0) + \sqrt2\big(\nabla h_5(\mathbf 0) + \nabla h_6(\mathbf 0)\big) =0.
\end{gather*}

In what follows we denote $(\lambda_1,\dots,\lambda_6) = (1,1,1,1,\sqrt2,\sqrt2)$ the coefficients in this dependence, and decompose the polynomials $\lambda_i h_i$  into their linear part (gradient),  quadratic part (Hessian) and cubic part. There is no part of degree zero since $h_i(\mathbf 0)=0$ by construction. That is to say:
	\[
		\lambda_i h_i (\mathbf t)= \underbrace{l_i(\mathbf t)}_{\text{linear}} + \underbrace{q_i(\mathbf t)}_{\text{quadratic}} + \underbrace{r_i(\mathbf t)}_{\text{cubic}}.
	\]
We now consider a positive constant $c\in \R_{> 0}$ (to be specified later) and define the function
	\begin{align*}
		h (\mathbf t)& =
		\sum_{i=1}^6  \left(c-\lambda_i \nabla h_i(\mathbf 0)\cdot \mathbf t\right)\ \lambda_i h_i (\mathbf t)
=		\sum_{i=1}^6 (c-l_i(\mathbf t)) (l_i(\mathbf t)+ q_i(\mathbf t) + r_i(\mathbf t)).
	\end{align*}

Observe that $\nabla h(\mathbf 0) = 0$, since:
\[
\nabla h(\mathbf 0) =  \sum_{i=1}^6 c \nabla l_i(\mathbf 0) = c  \sum_{i=1}^6 \lambda_i \nabla h_i(\mathbf 0) =0.
\]

\begin{lemma}
\label{lemma:small_c}
The Hessian $\nabla^2 h(\mathbf 0)$ is negative definite for any sufficiently small $c>0$.
\end{lemma}

\begin{proof}
The degree-two part of $h$ is
\[
\sum_{i=1}^6  (c q_i - l_i^2)=
c\sum_{i=1}^6 q_i  - \sum_{i=1}^6 l_i^2.
\]
For $c=0$ this equals $-\sum_{i=1}^6 l_i^2$, which is negative semi-definite with null-space equal to
\[
V = \left\{\mathbf v \in \R^8 :  \nabla h_i(\mathbf 0)\cdot \mathbf v = 0, i=1,\dots,6 \right\}.
\]
Thus, to prove the statement we only need to check that the quadratic form $\sum_{i=1}^6 q_i$ is negative definite when restricted to $V$. The following summarizes our computations in \textit{SageMath}, which prove that this is indeed the case.

The set $V$ is a three-dimensional linear subspace of $\R^8$ which admits the parametric form  $V = \{ \mathbf v (w_1,w_2,w_3) :  w_1,w_2,w_3 \in \mathbb R\}$, where
\begin{align*}
\mathbf v (w_1,w_2,w_3) := & \bigg( 1,0,0,0, \frac{\sqrt2}{2}, \frac{\sqrt2}{2}, -\frac{\sqrt2}{2}, \frac{\sqrt2-2}{2} \bigg) w_1 \\
 + & \bigg( 0,1,0, -\sqrt2, \frac{2-\sqrt2}{2}, \frac{2-\sqrt2}{2}, \frac{\sqrt2}{2}, \frac{2-3\sqrt2}{2} \bigg) w_2\\
 + & \bigg( 0,0,1, -1,1-\sqrt2, 1,0,-\sqrt2 \bigg) w_3.
\end{align*}

The Hessian of $\sum_{i=1}^6 q_i $ at $\mathbf 0$, restricted to the subspace $V$ and expressed in the coordinates $(w_1,w_2,w_3)$, turns out to be:
\[
\left[\
\sum_{i=1}^6
\
\begin{matrix}
{\partial^2  q_i(\mathbf v) }\\
\hline
{\partial w_j \partial w_k}
\end{matrix}
\ \ \right]_{j,k}=
\begin{bmatrix}
-\frac{19\sqrt2}{2} - 13 & -\frac{5\sqrt2}{2} - 4 &-\sqrt2 - 2\\
-\frac{5\sqrt2}{2} - 4 & -\frac{39\sqrt2}{2} - 27 & -16\sqrt2 - 22\\
-\sqrt2 - 2 & -16\sqrt2 - 22 & -19\sqrt2 - 26
\end{bmatrix},
\]
which is indeed negative definite.
\end{proof}

\begin{proof}[Proof of \Cref{theorem:6functionals}]
Let $c>0$ be such that the Hessian $\nabla^2 h(\mathbf 0)$ is negative definite. Such a $c$ exists by \Cref{lemma:small_c}.
Since $h(\mathbf 0)=0$ and $\nabla h (\mathbf 0) =0$, negative-definiteness of $\nabla^2 h(\mathbf 0)$ implies that there is a neighborhood $U_h$ of the origin such that $h$ is strictly negative in $U_h\setminus \{\mathbf 0\}$. On the other hand, there is another neighborhood $U_l$ of the origin in which all the multipliers $c- l_i$ are positive, since $c>0$ and $l_i(\mathbf 0) =0$.

Thus, for any $\mathbf t \in (U_h\cap U_l) \setminus \{\mathbf 0\}$ there is an $i$ such that
$\lambda_ih_i(\mathbf t) <0$; further intersecting with a neighborhood $U_M$ of $\mathbf 0$ where $\det M(\mathbf t) >0$, yields $f_i(\mathbf t) < 2+\sqrt2$ for any $\mathbf t \in (U_h\cap U_l \cap U_M) \setminus \{\mathbf 0\}$.
\end{proof}

\begin{remark}
This proof is a refinement of the Karush-Kuhn-Tucker (KKT) conditions for local optimality (see, for example, \cite[Theorem 14.19]{Griva_etal}).
We need this refined version as a preparation for the next section, where we compute an explicit neighborhood $U$ of $\bo$ such that the tetrahedron $\Delta$ is guaranteed to have width smaller than $2+\sqrt2$ with respect to $\Lambda(\bt)$ for every $\bt \in U \setminus \{\bo\}$.
\end{remark}

\section{An explicit neighborhood for maximality}
\label{sec:explicit}

We have seen that, among the lattices $\Lambda(\bt)$ for $\bt$ close to $\mathbf 0$, the width of $\Delta$ is largest with respect to $\Lambda(\mathbf 0)$. In this section, we construct an explicit neighborhood $U$ of the origin where this is achieved.

For this, consider the following system of strict inequalities in $\R^8$, where $i\in \{1,\dots,6\}$. In \eqref{width:comp:ineq},  $v$ runs over all vertices of $\Delta - \Delta$ other than $v_i$:
\begin{align}
\tag{i} \label{det:ineq} \det M(\bt) & > 0 &  & \text{(cubic)}\\
\tag{ii} \label{lin:bounds} l_i(\bt) & < c & & \text{(linear)}\\
\tag{iii} \label{width:comp:ineq} (v_i-v) M(\bt)^{\#} u_i & >0 &  & \text{(quadratic)}\\
\tag{iv} \label{hess:ineq} \nabla^2 h(\bt) & \phantom{>} \text{is negative definite} & & \text{(QMI)}
\end{align}

All these inequalities are satisfied at $\bt= \mathbf0$ when $c$ is chosen ``sufficiently small''. Hence, they are satisfied in a certain neighborhood $U$ of $\mathbf0$.
Our proof of \Cref{theorem:main} gives in fact the following stronger result:

\begin{theorem}
\label{thm:neighborhood}
If $c\in (0,\infty)$ and $U$ is a neighborhood of $\mathbf 0$ such that conditions {\rm (\ref{det:ineq})-(\ref{hess:ineq})} are met for every $\bt\in U$, then
\[
\width_{\Lambda(\mathbf 0)}(\Delta) > \width_{\Lambda(\mathbf t)}(\Delta)
\]
for every $\bt \in U \setminus\{\mathbf0\}$.
\end{theorem}

\begin{proof}
We can proceed as in the proof of \Cref{theorem:6functionals}, where the role of all inequalities except \eqref{width:comp:ineq} is explained. These other inequalities are needed to guarantee that the width of $\Delta$ with respect to the lattice functional $M(\bt)^{\#} u_i$ is given by the functional $f_i(\bt)$, since $f_i$ evaluates the functional $M(\bt)^{\#} u_i$ at the particular pair of vertices of $\Delta$ given by the corresponding $v_i \in \Delta - \Delta$. This was mentioned in \Cref{sec:setting}, where a neighborhood satisfying these inequalities was called $U_f$. The other three (sets of) inequalities correspond to the neighborhoods $U_M$, $U_l$ and $U_h$ in the proof of \Cref{theorem:6functionals}.
\end{proof}

\subsection{A change of variables}
\label{subsec:bs}

For the computation of the explicit neighborhoods we introduce new variables that are better adapted to the problem. We define
$\sh{j}$, $\sv{j}$ for $j=1,\dots,4$ by setting
\[
\begingroup 
\setlength\arraycolsep{2pt}
\begin{bmatrix}
\sh{1}\\
\sv{1}\\
\sh{2}\\
\sv{2}\\
\sh{3}\\
\sv{3}\\
\sh{4}\\
\sv{4}
\end{bmatrix}
\!=\!
\frac{1}{4}\!
\begin{bmatrix}
\sqrt2 \!-\! 1 &             -1 &              0 &              0 &              0 &              0 &              0 &              0\\
            -1 & 1 \!-\! \sqrt2 &              0 &              0 &              0 &              0 &              0 &              0\\
             0 &              0 &              1 & \sqrt2 \!-\! 1 &              0 &              0 &              0 &              0\\
             0 &              0 & \sqrt2 \!-\! 1 &             -1 &              0 &              0 &              0 &              0\\
             0 &              0 &              0 &              0 & 1 \!-\! \sqrt2 &              1 &              0 &              0\\
             0 &              0 &              0 &              0 &              1 & \sqrt2 \!-\! 1 &              0 &              0\\
             0 &              0 &              0 &              0 &              0 &              0 &             -1 & 1 \!-\! \sqrt2\\
             0 &              0 &              0 &              0 &              0 &              0 & 1 \!-\! \sqrt2 &              1
\end{bmatrix}\!
\begin{bmatrix}
t_{11}\\
t_{12}\\
t_{21}\\
t_{22}\\
t_{31}\\
t_{32}\\
t_{41}\\
t_{42}
\end{bmatrix}.
\endgroup
\]

In what follows, we denote by
\[
\bs:= \big(\sh{1}, \sv{1}, \sh{2}, \sv{2}, \sh{3}, \sv{3}, \sh{4}, \sv{4} \big),
\]
the vector of the new variables. These new variables $\bs$ have two advantages over the (projected) Cartesian variables $\bt$ used so far.

On the one hand, they respect the symmetry of the problem. Our tetrahedron $\Delta$ (and the system of inequalities we want to study) is invariant under the rotary reflection $(x,y,z) \to (y,-x,-z)$ in the original Cartesian coordinates. Indeed, this map sends the points $a_i$ and $p_i$ to $a_{i+1}$ and $p_{i+1}$, for each $i\in \{1,2,3,4\}$ and with indices taken modulo four. In the coordinates $\bt$ this isometry maps
\[
\left[
\begin{matrix}
t_{11}& t_{12}\\ t_{21}& t_{22}\\ t_{31}& t_{32}\\ t_{41}& t_{42}
\end{matrix}
\right]
\mapsto
\left[
\begin{matrix}
t_{22}& -t_{21}\\ t_{32}& -t_{31}\\ t_{42}& -t_{41}\\ t_{12}& -t_{11}
\end{matrix}
\right],
\]
whereas in the $\bs$ coordinates it simply maps each pair $\big(\sh{j}, \sv{j} \big)$ to $\big(\sh{j+1}, \sv{j+1} \big)$.

On the other hand, the $\bs$ coordinates are closely related to the barycentric coordinates along the facets of $\Delta$. Consider for example the facet spanned by $a_2$, $a_3$ and $a_4$, containing the point $p_1(\bt)$. Every point in the hyperplane containing that facet can be thought of in the $\bt$ coordinates as a $(t_{11},t_{12}) \in \R^2$. The change $(t_{11},t_{12}) \mapsto \big(\sh{1}, \sv{1} \big)$ makes the vertices of the facet have the following coordinates:
\[
\begin{array}{ccccc}
&&(t_{11},t_{12})&& \big(\sh{1}, \sv{1} \big)\\
\hline
a_2 &=& (-\sqrt2, 2+\sqrt2) &\mapsto& (-1,0), \\
a_3 &=& (-2-\sqrt2, -\sqrt2) &\mapsto& (0,1), \\
a_4 &=& (\sqrt2, -2-\sqrt2) &\mapsto& (1,0).
\end{array}
\]
That is to say: if $(b_2,b_3,b_4)$ are the barycentric coordinates of a point in that hyperplane with respect to the affine basis $(a_2,a_3,a_4)$ then we have that
\begin{equation}
\label{eq:barycentric}
\sv{1} = b_3 = \frac12 +\frac{t_{13}}{4+2\sqrt2},
\qquad \text{and} \qquad
\sh{1} = b_4-b_2.
\end{equation}
This explains the notation $\sh{}$ and $\sv{}$:
$\sv{}$ is constant along horizontal lines in the facet, and $\sh{}$ is constant along ``steepest'' lines in the facet, see Figure~\ref{fig:s-coordinates}.

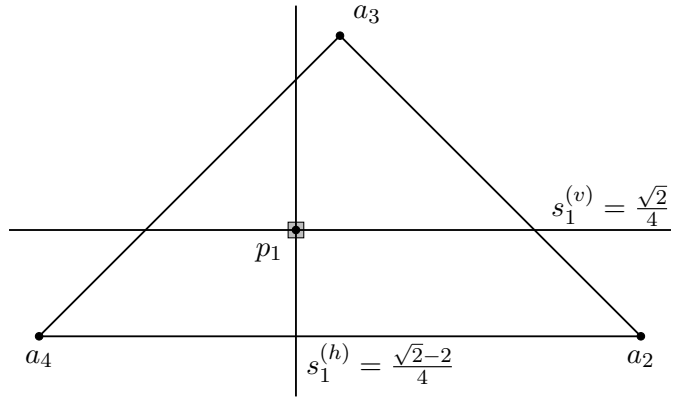
\begin{figure}[ht!]
\begin{tikzpicture}[bnode/.style={circle, inner sep=1.5pt, fill=black},
wnode/.style={circle, inner sep=1pt, fill=black},
scale = 4]
\clip(-1.2,-0.2) rectangle (1.2,1.2);
\draw[line width=0.7pt] (-1.,0.) -- (1.,0.) -- (0.,1.) -- cycle;
\draw[fill=lightgray] (-0.1724,0.3276) -- (-0.1724,0.3796) -- (-0.1204,0.3796) -- (-0.1204,0.3276) -- cycle;
\draw[line width=0.7pt] (-0.1464,-0.2)-- (-0.1464,1.1);
\draw[line width=0.7pt] (-1.1, 0.3536)-- (1.1, 0.3536);
\node[draw,wnode, label=below left:{$p_1$}] () at (-0.1464, 0.3536) {};
\node[draw,wnode, label=below:{$a_4$}] () at (-1, 0) {};
\node[draw,wnode, label=below:{$a_2$}] () at (1, 0) {};
\node[draw,wnode, label=above right:{$a_3$}] () at (0, 1) {};
\node[label=above:{$\sv{1}=\frac{\sqrt2}{4}$}] () at (0.9, 0.30) {};
\node[label=right:{$\sh{1}=\frac{\sqrt2-2}{4}$}] () at (-0.18, -0.09) {};
\end{tikzpicture}
\caption{An illustration of the coordinates $\sh{1}$ and $\sv{1}$. The gray square represents the $L_\infty$-ball of radius $0.02614$, the bound in \Cref{thm:explicit}}
\label{fig:s-coordinates}
\end{figure}

We are going to abuse notation and keep the notation of conditions (\ref{det:ineq})-(\ref{hess:ineq}) for the corresponding conditions in the new variables. Thus, for example, we write $M(\bs)$ for $M(\bt(\bs))$, and we write $\nabla^2 h(\bs)$ for
$\left(\partial^2 h(\bt(\bs)) / \partial \bt^2\right)$ (the Hessian of $h$ with derivatives with respect to the $\bt$ coordinates, but expressed in the $\bs$ coordinates). In \Cref{sec:det:ineq,sec:width:comp:ineq,sec:lin:bounds,sec:hess:ineq} we look separately at the four (sets of) inequalities and find that they are satisfied in the following neighborhoods, all expressed as $L_\infty$ balls in the $\bs$ coordinates:

\begin{itemize}
\item[(i)] is a single inequality involving a polynomial of degree three; in \Cref{sec:det:ineq}
(\Cref{coro:determinant}), using a geometric argument, we show that it holds for every $\bs$ with  $\|\bs\|_\infty < \frac14(\sqrt2-1) \approx 0.1036$.
\item[(ii)] are six linear inequalities; \Cref{sec:lin:bounds} is devoted to proving that, taking $c=9.75$, these are satisfied
for $\|\bs\|_\infty < 0.02614$.
\item[(iii)] consists of $6 \times 11$ inequalities (since $\Delta - \Delta$ has $12$ vertices); these are simultaneously verified when $\|\bs\|_\infty <  0.04423$, as shown in \Cref{sec:width:comp:ineq} (\Cref{coro:quadratic}).
\item[(iv)] is a strict quadratic matrix inequality (QMI); in \Cref{sec:hess:ineq} we show that, for the same $c=9.75$, it holds when $\|\bs\|_\infty <   0.02646$.
	\end{itemize}

\begin{remark}
The inequalities \eqref{lin:bounds} and \eqref{hess:ineq}, hence the bound obtained from them,  depend on $c$. Moreover, these are the two inequalities were our bounds are worse.
The dependence on $c$ for the bound in \eqref{lin:bounds} is obviously proportional to $c^{-1}$ and we can compute it quite explicitly (see \Cref{coro:linear}). For \eqref{hess:ineq}, we have computed the bound for several values of $c$, finding that $c=9.75$ is (very close to) the optimum. See \Cref{rem:why-9.75} for more details on this.
\label{rem:c-dependence}
\end{remark}

This allows us to conclude:
\begin{theorem}
\label{thm:explicit}
The width of $\Delta$ with respect to $\Lambda(\mathbf 0)$ is strictly larger than with respect to any other lattice $\Lambda(\bs)$ with $\bs\ne \bo$ and $\|\bs\|_\infty <0.02614$.
\end{theorem}

\begin{proof}[Proof of \Cref{thm:barycentric}]
Let $p'_1, p'_2, p'_3, p'_4$ be points in the corresponding facets of $\Delta$. The relation between barycentric coordinates and $\bs$ coordinates expressed in \eqref{eq:barycentric} implies that changing barycentric coordinates by at most $\varepsilon$, the $\bs$ coordinates change by at most $2\varepsilon$.
\end{proof}

\subsection{The determinant condition (\ref{det:ineq})}
\label{sec:det:ineq}

Let us first compute the $\bs$-coordinates of the initial lattice points $p_i = p_i(\bo)$; we do this for $p_1$, but the result is the same for the other points, by symmetry. We have that $p_1=(-1,-1,-1)$ so that $t_{11}=t_{12}=-1$ and
\[
\begingroup 
\setlength\arraycolsep{2pt}
\begin{bmatrix}
\sh{1}\\
\sv{1}
\end{bmatrix}
\!=\!
\frac{1}{4}\!
\begin{bmatrix}
\sqrt2 \!-\! 1 &             -1\\
            -1 & 1 \!-\! \sqrt2
\end{bmatrix}\!
\begin{bmatrix}
-1\\
-1
\end{bmatrix}
\!=\!
\frac{1}{4}\!
\begin{bmatrix}
             2 \!-\!  \sqrt2 \\
             \sqrt2
\end{bmatrix}
\approx
\begin{bmatrix}
0.15 \\
0.35
\end{bmatrix}.
\endgroup
\]

For each $i\in \{1,2,3,4\}$, let $F_i$ be the facet of $\Delta$ containing $p_i = p_i(\bo)$ and let $T_i$ be the open triangle with vertices at the mid-points of edges of $F_i$. In barycentric coordinates, $T_i$ is the set of points with all coordinates less than 1/2.
By what we said before, the $\bs$ coordinates that we have got for $p_1$ correspond to the barycentric coordinates $\frac14(1,\sqrt2,3-\sqrt2)\approx (0.25, 0.35, 0.40)$. Thus, the point $p_i$ is in $T_i$.

\begin{lemma}
\label{lemma:determinant}
Every $4$-tuple $(q_1,q_2,q_3,q_4)\in T_1\times T_2\times T_3\times T_4$ has the same orientation (i.e., the same sign of the determinant) as $(p_1,p_2,p_3,p_4)$.
\end{lemma}

\begin{proof}
The result follows immediately if we prove that no $(q_1,q_2,q_3,q_4)\in T_1\times T_2\times T_3\times T_4$ produces a coplanar 4-tuple. To show this, suppose by way of contradiction that it does, so that there are coplanar $q_i\in T_i$, $i=1,2,3,4$.

First observe that the fact that the $q_i$'s lie in the respective $T_i$'s implies that, for each choice of $(i,j) \in \binom{[4]}2$ there is an affine functional $f_{ij}$ that is positive on $q_i$ and $q_j$ and negative on the other two points $q_k$ and $q_l$. Indeed, this is the functional that bisects $\Delta$ taking value $1$ at the vertices $a_k$ and $a_l$ and $-1$ at $a_i$ and $a_j$.

On the other hand, coplanarity implies the existence of an affine dependence
\[
\mu_1 q_1 + \mu_2 q_2 + \mu_3 q_3 + \mu_4 q_4 =0, \qquad \text{with }
\mu_1  + \mu_2  + \mu_3  + \mu_4  =0.
\]
Now, since the $q_i$'s lie in the relative interior of different facets of $\Delta$, none of them is in the convex hull of the rest. This implies that two of the $\mu_r$'s, say $\mu_i$ and $\mu_j$, are positive and the other two are negative.  Then the contradiction is that
\begin{align*}
0 &= f_{ij}(\mu_1 q_1 + \mu_2 q_2 + \mu_3 q_3 + \mu_4 q_4) \\
&= \mu_1 f_{ij}(q_1) + \mu_2 f_{ij}(q_2) + \mu_3 f_{ij}(q_3) + \mu_4 f_{ij}(q_4) >0.
\qedhere
\end{align*}
\end{proof}

\begin{corollary}
\label{coro:determinant}
If $\bs \in \R^8$ is chosen such that $\|\bs\|_\infty < \frac14(\sqrt2-1) \approx 0.1036$ then the determinants $\det M(\mathbf 0)$ and $\det M(\bs)$ have the same sign.
\end{corollary}

\begin{proof}
We simply need to check that the open $\infty$-ball around, say, $p_1$ and of radius $\frac14(\sqrt2-1)$ does not intersect the boundary of the triangle $T_1$. In the coordinates $\big(\sh{1},\sv{1}\big)$ the vertices of $T_1$ are $(0,0)$, $\frac12(-1,1)$ and $\frac12(1,1)$, so that its facet description is
\[
\sv{1} <\frac12,
\qquad
\sv{1} +\sh{1} >0,
\qquad
\sv{1} -\sh{1} >0.
\]
The vertices of the $\infty$-ball in the statement have
\[
\big(\sh{1},\sv{1}\big) = \frac{1}{4} \big(2 - \sqrt2 \pm (\sqrt2-1), \ \sqrt2 \pm(\sqrt2-1) \big),
\]
and the four of them satisfy the three non-strict inequalitites. (Strictness is not needed since the ball is open).
\end{proof}

\subsection{The linear inequalities (\ref{lin:bounds})}
\label{sec:lin:bounds}

To address the inequalities $l_i(\bs) \le c$ let us first explicitly write the linear functions $l_i$ in the original $\bt$ coordinates. They are essentially the same as the gradients in \Cref{eq:nablas}, except they have to be multiplied respectively by the coefficients  $(\lambda_1,\dots,\lambda_6) = (1,1,1,1,\sqrt2,\sqrt2)$:
\begin{align*}
\begin{array}{ll}
 l_1(\bt) = \bt \cdot \big(4(-1,1,-1,-2,0,0,-2,1) &+\ 2\sqrt2(-2,0,1,-1,0,0,-3,1)\big)  \\
 l_2(\bt) = \bt \cdot \big(4(-2,1,0,0,1,2,1,1) &+\ 2\sqrt2(-1,-1,0,0,1,3,0,2) \big)  \\
 l_3(\bt) = \bt \cdot \big(4(0,0,2,-1,1,-1,1,2) &+\ 2\sqrt2(0,0,3,-1,2,0,-1,1)\big)  \\
 l_4(\bt) = \bt \cdot \big(4(-1,-2,-1,-1,2,-1,0,0) \hspace{-5pt}\null&+\ 2\sqrt2(-1,-3,0,-2,1,1,0,0)\big)  \\
 l_5(\bt) = \bt \cdot \big(16(1,0,0,0,-1,0,0,0)  &+\ 8\sqrt2(1,0,-1,0,-1,0,1,0)\big)  \\
 l_6(\bt) = \bt \cdot \big(16(0,0,0,1,0,0,0,-1) &+\ 8\sqrt2(0,1,0,1,0,-1,0,-1) \big)
\end{array}
\end{align*}

From here, we obtain their expression in the $\bs$ coordinates, multiplying by the inverse change of coordinates. Observe that, as expected, these coordinates highlight the symmetry of the problem, acting as a shift of  the $\bs$ coordinates by two places:
\begin{align*}
\arraycolsep1pt
\renewcommand\arraystretch{1.2}
\begin{array}{ll}
	\text{ \normalsize $l_1(\bs) \!= \ 8\bs \cdot$}
		\big(  (-2, 2, -1, 3,\ \ 0, 0,\ \ 3, 3)     & \ +\   \sqrt2 (-1, 1, -1, 1,\ \ 0, 0,\ \ 2, 2) \big)  \\
	\text{ \normalsize $l_2(\bs) \!= \ 8\bs \cdot$}
		\big(  (-1, 3,\ \ 0, 0,\ \ 3, 3, -2, 2)     & \ +\    \sqrt2 (-1, 1,\ \ 0, 0,\ \ 2, 2, -1, 1) \big)  \\
	\text{ \normalsize $l_3(\bs) \!= \ 8\bs \cdot$}
		\big(  (\ \ 0, 0,\ \ 3, 3, -2, 2, -1, 3)     & \ +\   \sqrt2 (\ \ 0, 0,\ \ 2, 2, -1, 1, -1, 1) \big)  \\
	\text{ \normalsize $l_4(\bs) \!= \ 8\bs \cdot$}
		\big(  (\ \ 3, 3, -2, 2, -1, 3,\ \ 0, 0)     & \ +\   \sqrt2 (\ \ 2, 2, -1, 1, -1, 1,\ \ 0, 0) \big)  \\
	\text{ \normalsize $l_5(\bs) \!=\! 16\bs \cdot$}
		\big(  ( 1,-3, -1,-1,  1,-3, -1,-1)     & \ +\   \sqrt2 ( 1,-2, -1, 0,  1,-2, -1, 0)\big)  \\
	\text{ \normalsize $l_6(\bs) \!=\! 16\bs \cdot$}
		\big(  (-1,-1,  1,-3, -1,-1,  1,-3)     & \ +\   \sqrt2 (-1, 0,  1,-2, -1, 0,  1,-2)\big)  \\
\end{array}
\end{align*}

Recall that the $L_\infty$ distance of a hyperplane  $\sum_i a_i x_i = c$ to the origin is simply
$|c|/\sum_i |a_i|.$
In our case, the sum of absolute values of coefficients is $112 + 64\sqrt2 \approx 202.51$ for $l_1,\dots,l_4$ and it is $192 + 128\sqrt2\approx 373.02$ for $l_5$ and $l_6$. Thus:

\begin{corollary}
\label{coro:linear}
If $c\in (0,\infty)$ and $\bs \in \R^8$ is chosen such that $\|\bs\|_\infty < c / (128\sqrt2+192)$, then the
six inequalities $l_i(\bs) < c$ are satisfied.
\end{corollary}

In \Cref{fig:coeffBounds}, the blue curve represents the above bound for $c$ in the interval $[7,12]$.
For $c=9.75$ we obtain the bound $9.75 / (128\sqrt2+192)\approx 0.02614$.

\subsection{The inequalities (\ref{width:comp:ineq})}
\label{sec:width:comp:ineq}

The inequalities  $(v_i-v) M(\bs)^{\#} u_i >0$ are 66 polynomial inqualities of degree two. To find a neighborhood where these inequalities hold we use the following criterion, involving only the coefficients of the polynomials:

\begin{proposition} \label{thm:root_pos_poly}
Let $f(x_1,\dots,x_n)$ be a polynomial in several variables with $f(0,\dots,0)\ne 0$. Let $f_0(x)$ be the univariate polynomial obtained from $f$ as follows:
\begin{itemize}
\item Change all coefficients to their absolute values, except the constant term that is changed to minus its absolute value.

\item Make all variables equal to a single one, $x$.
\end{itemize}

Then, no zero of $f$ has $L_\infty$-norm smaller than the unique positive root of $f_0$.
\end{proposition}

\begin{proof}
$f_0$ is a univariate polynomial with negative coefficient in degree zero and positive coefficient in all other degrees. From this, Descartes' rule of signs implies that it has a unique positive root, that we denote by $r_0$.

For each $J\in \N^n$, let $c_J$ denote the coefficient of multidegree $J$ in $f$. Let $a_j$ denote the sum of absolute values of all $c_J$ for each fixed total degree $j\in \N$. Observe that $f_0(x) = \sum_{j=1}^\infty a_j\, x^j -a_0$.
If $z = (z_1,\dots,z_n)$ is a zero of $f$ with $r=\max_i |z_i|$, we have
\[
0=f(z_1,\dots,z_n) = c_0 + \sum_{J\ne 0} c_J z^J,
\]
so that
\[
a_0 = |c_0| = \bigg| \sum_{J\ne 0} c_J z^J \bigg|
\le \sum_{j=1}^\infty a_j\, r^j  = f_0(r) + a_0.
\]
Thus, $f_0(r) \ge  0$, which implies that $r\ge r_0$.
\end{proof}

Applying \Cref{thm:root_pos_poly} to each of the 66 quadratic polynomials and using \textit{SageMath} \cite{Sage} to execute the computations, we found the following bound:

\begin{corollary}
\label{coro:quadratic}
If $\bs \in \R^8$ is chosen such that $\|\bs\|_\infty < 0.04423$, then the  quadratic inequalities \eqref{width:comp:ineq} are satisfied.
\end{corollary}

\subsection{The Hessian (\ref{hess:ineq})}
\label{sec:hess:ineq}

Let $U$ be a connected open subset in the space of real symmetric $n\times n$ matrices. Suppose that the determinant never vanishes in $U$ and that $U$ contains a negative definite matrix $M_0$. Then, all matrices in $U$ are negative definite, because all eigenvalues of $M_0$ are strictly negative and, by continuity of eigenvalues, every matrix in $U$ has all its eigenvalues strictly negative.

Hence, what we need to compute in this section is an open ball $U$ in the $\bs$ coordinates such that $\nabla^2h(\mathbf s)$ has strictly positive determinant for every $\bs \in U\setminus\{\bo\}$.
Once a value of $c$ is chosen, the entries of $\nabla^2h(\mathbf s)$ are quadratic polynomials in $\bs$, so its determinant  is a polynomial of degree $16$ in $8$ variables. Thus, in order to compute a neighborhood of $\bo$ in which the determinant of $\nabla^2h(\mathbf s)$ stays positive we can apply \Cref{thm:root_pos_poly} to this polynomial. Doing this with $c=9.75$ we found that:

\begin{corollary}
\label{coro:hessian}
If $c=9.75$ and $\|\bs\|_\infty  < 0.02646$, then $\nabla^2h(\mathbf s)$ is negative definite.
\end{corollary}

The computations were performed in \textit{SageMath} \cite{Sage} in exact arithmetic and they took about 14 hours with a computer with 8GB of RAM.

\begin{remark}[{\it Why $c=9.75?$}]
\label{rem:why-9.75}
Using \textit{SageMath}, we experimentally checked that the Hessian $\nabla^2h(\mathbf 0)$  stays negative definite for $c \in [0,13.254)$. Our first trial for $c$ was $c=7$, which gave a neighborhood for negative-definiteness much bigger than the one we have for that value of $c$ in (\ref{lin:bounds}). Since the latter increases with $c$ (see \Cref{coro:linear}, or the blue line in \Cref{fig:coeffBounds}), we recomputed the bound for the Hessian with bigger values of $c$, obtaining that the bound for the Hessian decreases with $c$ and meets the one for (\ref{lin:bounds}) very close to $c=9.75$. Finding the exact optimal $c$ is meaningless, since the bound provided by \Cref{thm:root_pos_poly} is (expected to be) much smaller than the smallest $L_\infty$-norm among the roots of the multivariate polynomial under study (observe that we are speaking of a polynomial of degree 16 in 8 variables).

The computational results are given in Table~\ref{tab:list} and \Cref{fig:coeffBounds}.
\end{remark}

\begin{table}[ht!]
\small
\begin{tabular}{c|ccccccc}
 c& 7 & 8 & 9 & 9.75 & 10 & 11 & 12 \\
\hline
(\ref{det:ineq})        & = & = & = & \textbf{0.10355} & = & = & = \\
(\ref{lin:bounds})      & 0.01877 & 0.02145 & 0.02413 & \textbf{0.02614} & 0.02681 & 0.02949 & 0.03217 \\
(\ref{width:comp:ineq}) & = & = & = & \textbf{0.04423} & = & = & = \\
(\ref{hess:ineq})       & 0.03185 & 0.03028 & 0.02834 & \textbf{0.02646} & 0.02571 & 0.02123 & 0.01501 \\
\end{tabular}\\[2mm]
\caption{Bounds for different values of $c$} \label{tab:list}
\end{table}

\begin{figure}[ht!]
  \centering
\includegraphics[width=\textwidth]{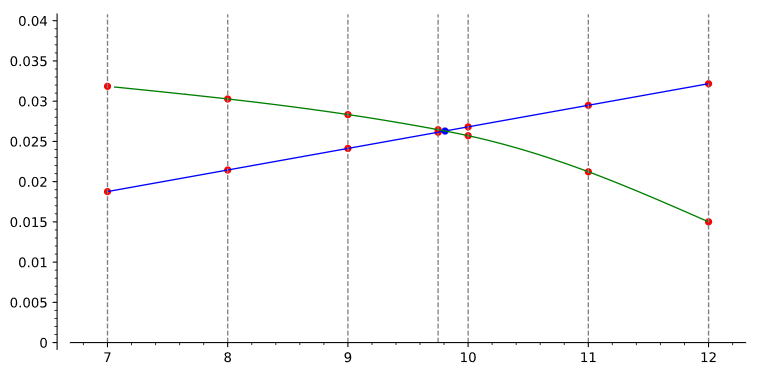}
\caption{The two bounds that depend on $c$, for different values of $c$ on the $x$-axis.
Blue: the bound for (\ref{lin:bounds}), linear in $c$.
Green: the bound for (\ref{hess:ineq}), experimental.
}
\label{fig:coeffBounds}
\end{figure}

\section{A search space of lattice-width maximizers}
\label{sec:global}

In this section, we give geometric properties that a hollow three-dimensional convex body must satisfy to be a lattice-width maximizer. This means that we are limiting the search space for possible counterexamples to \Cref{conj:maximal}.

\subsection{Relevant inequalities}

We consider the following invariants of a convex body $K$ in dimension three (the relations among these are the topic of \cite{KanLov}): the covering minima $\mu_i(K)$, the successive minima $\lambda_i(K-K)$ of the difference body $K-K$ of $K$, where $i \in\{1,2,3\}$, the (Euclidean) volume
\(\vol(K)
\)
of $K$,  the volume of the difference body
\(\vol(K-K)
\)
and the volume of the dual of the difference body
\(
\vol((K-K)^\ast).
\)

Covering and successive minima depend on an ambient lattice, which in this section we assume to be the standard lattice $\Z^3$.
The $i$-th \textit{covering minimum} $\mu_i(K)$ of $K$ is the minimal $\mu>0$ such that $\mu K + \Z^3$ has non-empty intersection with every $(3-i)$-dimensional affine subspace of $\R^3$. The \textit{successive minimum} $\lambda_i(C)$ of an origin-symmetric convex body $C$ is the minimal $\lambda>0$ such that the lattice vectors in $\lambda C$ span a vector space of dimension at least $i$. It directly follows from these definitions that
\begin{align*}
0 &< \lambda_1(K-K) \le \lambda_2(K-K) \le \lambda_3(K-K),
\\ 0 & < \mu_1(K) \le \mu_2(K) \le \mu_3(K).
\end{align*}
Both covering minima and successive minima are related to the lattice width $w(K) := \width_{\Z^3}(K)$, via the following equalities from \cite{KanLov}:
\[
w(K) = \frac1{\mu_1(K)} = \lambda_1((K-K)^*).
\]
The following inequalities are known:
\begin{align}
\mu_3(K) & \le \mu_2(K) + \lambda_1(K-K), \label{mu3:eq}
\\	\mu_2(K) & \le (1+ 2 / \sqrt{3} ) \mu_1(K), \label{mu21:eq}
\\ \lambda_1(K-K)^3 \vol(K-K)  & \le 8, \label{mink:2}
\\  \vol((K-K)^\ast) & \le 8 \mu_1(K)^3. \label{mink:dual}
\end{align}
Here, \eqref{mu3:eq} and \eqref{mu21:eq} can be found in \cite{KanLov}; \eqref{mu21:eq} is a re-formulation of the exact flatness theorem in dimension two of Hurkens \cite{Hurkens}; \eqref{mink:2} and \eqref{mink:dual} are Minkowski's first theorem (see \cite{GruLek}) applied respectively to $K-K$ and $(K-K)^*$.

We also make use of the following inequalities, not related to the lattice:
\begin{align}
 8 \vol(K) & \le \vol(K-K), \label{brunn:mink:eq} && \text{(Brunn-Minkowski)}
\\  \vol(K-K) & 	\le 20 \vol(K), \label{diff:body:eq} && \text{(Rogers-Shephard)}
\\ \frac{32}{3} & \le \vol(K-K) \cdot \vol((K-K)^\ast). \label{mahler:vol:eq} && \text{(Mahler's inequality)}
\end{align}
Inequalities \eqref{brunn:mink:eq} and \eqref{diff:body:eq} are well-known (see, for example, \cite{GruLek}).
Equality is attained in \eqref{brunn:mink:eq} if and only if $K$ is centrally symmetric,
and in \eqref{diff:body:eq} if and only if $K$ is a tetrahedron.
 Inequality
\eqref{mahler:vol:eq} is the three-dimensional case of Mahler's conjecture, and has recently been established in \cite{IriShi}. It is an equality when $K-K$ and its polar are an axis-parallel cube and octahedron.

We are particularly interested in three-dimensional hollow convex bodies. For these, we obtain the following inequalities:
\begin{lemma}\label{lem:derived:ineq}
	Let $K$ be a hollow convex body in dimension three. Then
	\begin{align}
	1 & \le \mu_3(K) \label{hollow_diseq}
	\\1 & \le (1 + 2 /\sqrt{3}) \mu_1(K)  + \lambda_1(K-K)  \label{deriv_eq:1}
	\\ 4 & \le 3 \mu_1(K)^3 \vol(K-K). \label{deriv_eq:2}
	\end{align}
\end{lemma}
\begin{proof}
The inequalities are direct implications of the inequalities and remarks given above:
\eqref{hollow_diseq} follows directly from the definition, while \eqref{deriv_eq:1} is obtained by combining \eqref{mu3:eq} and \eqref{mu21:eq} with \eqref{hollow_diseq}, and \eqref{deriv_eq:2} by combining \eqref{mahler:vol:eq} and \eqref{mink:dual}.
\end{proof}

\subsection{Bounds on volume and width}
Suppose now that $K$ is a maximizer for the lattice width among three-dimensional hollow convex sets. Then $w(K)=\mu_1^{-1} \ge 2 + \sqrt{2}$. Hence,
	\begin{align*}
	1 & \stackrel{\eqref{deriv_eq:1}}{\le}\left( 1 + \frac{2}{\sqrt{3}}\right) \mu_1(K) + \lambda_1(K-K)
	\le \frac{1 + 2/\sqrt{3}}{2 + \sqrt{2}} + \lambda_1(K-K),
	\end{align*}
	which gives a lower bound on the first successive minimum of $K-K$:
	\begin{align}\label{first_success:eq}
	\lambda_1(K-K) \ge 1 - \frac{1 + 2 / \sqrt{3}}{2 + \sqrt{2}}.
\end{align}

The following statement bounds width and volume of maximizers, combining ideas from
\cite{AverkovKrumpelmannWeltge,Dash_etal,IglSan}.

\begin{theorem} \label{thm:maximizer:properties} Let $K$ be a maximizer for the lattice width among three-dimensional hollow convex sets. Then
	\begin{align*}
	3{.}414 < & w(K) <  3{.}972 & &\text{and} &	2{.}653 < & \vol(K) < 19{.}919.
	\end{align*}
\end{theorem}

\begin{proof}
	The lower bound on $w(K)$ is provided by rounding the lower bound $w(K) \ge 2 + \sqrt{2}$, which follows from the existence of a hollow tetrahedron with  lattice width $2 + \sqrt{2}$ \cite[Theorem 5.2]{CodSan}.  The upper bound follows from the same chain of estimates as in \cite{Dash_etal}, except our use of Mahler's inequality, proved recently, gives a better bound than the one stated there:
	\begin{align*}
	1 & \stackrel{\eqref{deriv_eq:1}}{\le} (1+ 2/\sqrt{3}) \mu_1(K) + \lambda_1(K-K)
	\\ & \stackrel{\eqref{mink:2}}{\le}  (1+ 2/\sqrt{3}) \mu_1(K) + \frac{2}{\vol(K-K)^{1/3}}
	\\ & \stackrel{\eqref{deriv_eq:2}}{\le}  (1+ 2/\sqrt{3}) \mu_1(K) + 2 \left(\frac{3}{4}\right)^{1/3} \mu_1(K).
	\end{align*}
	This gives $w(K) = \frac{1}{\mu_1(K)} \le 1 + 2 /\sqrt{3} + 2 \left(\frac{3}{4}\right)^{1/3} < 3{.}972$.
	
	The bounds on $\vol(K)$ are derived from the estimates
\[
\vol(K) \!\stackrel{\eqref{diff:body:eq}}{\ge}\! \frac{1}{20} \vol(K-K) \!\stackrel{\eqref{deriv_eq:2}}{\ge}\! \frac{1}{20} \cdot \frac{4}{3 \mu_1(K)^3} \!=\! \frac{1}{15} w(K)^3 \!\ge\! \frac{1}{15} (2 + \sqrt{2})^3,
\]
	and
	\begin{align*}
	\vol(K)  \stackrel{\eqref{brunn:mink:eq}}{\le}  \frac{1}{8} \vol(K-K) \stackrel{\eqref{mink:2}}{\le}  \frac{1}{\lambda_1(K-K)^3}
	\stackrel{\eqref{first_success:eq}}{\le} \frac{1}{\left( 1 - \frac{1 +2 /\sqrt{3}}{2 + \sqrt{2}}\right)^3}.
	\end{align*}
\end{proof}

\subsection{Lattice polytopes inscribed in maximizers}

We say that a lattice polytope $P$ is \emph{inscribed} in a hollow polytope $K$ if $P$ contains at least one lattice point from the relative interior of each facet of $K$. Every width maximizer has such a (perhaps not unique) inscribed polytope, by the following result of  Lov\'asz \cite{Lov} (see \cite{Ave} and \cite{BCCZ} for complete proofs):

\begin{proposition}
\label{prop:lovasz}
Every maximal hollow convex set $K$ is a polyhedron and has at least one lattice point in the relative interior of each facet.
\end{proposition}

It is quite natural to approach the width maximization problem for hollow convex sets by distinguishing different choices of such inscribed polytopes and handling each case separately. This is in fact the approach that was used by Hurkens \cite{Hurkens} to settle the case of dimension two and we have also implicitly relied on it, although in our local situation $P$ is fixed; it is the unimodular tetrahedron $\conv(p_1,\dots,p_4)$.

It is clear that we can always choose the inscribed polytope $P$ to be an \emph{empty lattice polytope}, that is to say, a lattice polytope with no lattice point other than its vertices.
This happens, for example, if we pick a single lattice point from the interior of each facet of $K$ and let $P$ be their convex hull.
The fact that all empty lattice $3$-polytopes have width one allows us to provide a finite search space for the inscribed lattice polytopes that can occur in a width maximizer:

\begin{theorem} \label{thm:max:inscr:properties}
	Let $K$ be a width maximizer among hollow convex $3$-bodies and let $P$ be an empty lattice polytope inscribed in $K$. Then, up to unimodular equivalence, $P$ is either the square $\conv(0,e_1,e_2, e_1 + e_2)$ or a three-dimensional empty lattice polytope satisfying
	\[
	\vol(P) \le  \frac{22}{3}.
	\]
	Furthermore, if $P$ is a tetrahedron, then
	\[
	\vol(P) \le \frac{17}{6}.
	\]
\end{theorem}
\begin{proof}
	Any three points from the relative interiors of three distinct facets of $P$ are not collinear. Hence, $P$ has dimension at least two. Since $K$ has at least four facets, $P$ has at least four vertices.
	
	If the dimension of $P$ is two then $P$ is the square $\conv(0,e_1,e_2,e_1+e_2)$, since this the unique (up to affine unimodular transformation) empty lattice polygon with four or more vertices.

	If $P$ is three-dimensional, then we use Howe's theorem \cite{Scarf}, which states that every empty $3$-polytope has width $1$. This means that $P-P$ intersects only three consecutive layers of the lattice $\Z^3$ and, moreover, every smaller copy $\lambda (P-P)$ with $0 < \lambda < 1$ intersects at most just one lattice layer. 	This implies $\lambda_3 (P-P) = 1$. Using the second theorem of Minkowski \cite{GruLek}, we obtain
	\begin{align*}
		\vol(P-P) & \le \frac{8}{\prod_{i=1}^3 \lambda_i(P-P)}  & & \text{( Minkowski's 2nd theorem)}
		\\ &
			= \frac{8}{\prod_{i=1}^2 \lambda_i(P-P)} & & \text{(since $\lambda_3(P-P)=1$)}
		\\ & \le \frac{8}{\lambda_1(P-P)^2} & & \text{(since $\lambda_2(P-P) \ge \lambda_1(P-P)$)}.
	\end{align*}
	In view of $P \subseteq K$, this yields
	\[
	\vol(P) \stackrel{\eqref{brunn:mink:eq}}{\le} \frac{\vol(P-P)}{8} \leq \frac{1}{\lambda_1(P-P)^2} \le \frac{1}{\lambda_1(K-K)^2} \stackrel{\eqref{first_success:eq}}{\le} \frac{1}{( 1 - \frac{1 +2 /\sqrt{3}}{2 + \sqrt{2}})^2}.
	\]

	Since $P$ is lattice polytope, $6 \vol(P)$ is an integer value. Thus, rounding appropriately, we obtain the desired upper bound. If $P$ is a tetrahedron, then the above estimates can be improved by taking into account the equality $\vol(P-P) = 20 \vol(P)$, attained in \eqref{diff:body:eq}. This results into an improved bound on $\vol(P)$.
\end{proof}

Theorem~\ref{thm:maximizer:properties} allows to split the problem of detecting the maximum lattice width of hollow three-dimensional convex sets  into finitely many cases, as one can fix one of the finitely many possible inscribed polytopes $P$ and then maximize the lattice width among hollow convex sets with the given inscribed set $P$. It would be nice to rule out $P$ being a square, which is the only case when the inscribed polytope is two-dimensional, but currently we do not know how to handle this case.

When a \emph{full-dimensional} $P$ is fixed, the upper bound on the volume of $K$ in Theorem~\ref{thm:max:inscr:properties} allows to determine a bounding region $B$ (say, a box)  that depends only on $P$, in which $K$ is necessarily contained. By providing a bounding region for $K$, we get rid of the necessity to keep track of all (infinitely many) lattice points while expressing the hollowness of $K$ algebraically. Thus, the property we want to verify (that all three-dimensional convex sets with a fixed three-dimensional inscribed empty polytope $P$ have lattice width at least $2 + \sqrt{2}$) can be phrased in the first-order language of the real algebra. We thus conclude that, theoretically, such a property is decidable via quantifier-elimination algorithms for the first-order real-algebra sentences \cite[pp. 22--29]{BPR}. However, since the first-order sentences would be extremely complex and since the quantifier-elimination algorithms are extremely slow, such a brute-force approach is doomed to failure in practice.

Nevertheless, the above comments suggest that there might exist a reasonable way to reduce the problem of determination of the flatness constant in dimension three into a purely algebraic problem in terms of real variables and a system of polynomial inequalities. In contrast to the two-dimensional case, we do not expect however that our problem in dimension three can be solved without computer, because it is very likely that one would be forced to consider a large number of different cases and deal with rather complex algebraic problems in each of these cases.

\end{document}